\documentclass[11pt, oneside]{article}   	
\usepackage{geometry}                		
\usepackage{graphicx}				
\usepackage{amssymb, amsmath, amscd}
\usepackage{mathtools}
\mathtoolsset{showonlyrefs=true}		

\usepackage{amsthm}
\theoremstyle{plain}
\newtheorem{thm}{Theorem}[section]
\newtheorem{prop}{Proposition}[section]

\newtheorem{lem}{Lemma}[section]

\makeatletter						
\@addtoreset{equation}{section}

\makeatother
\def\RR{\mathbb{R}}
\def\SS{\mathbb{S}}
\def\eps{\varepsilon}
\def\inn<#1>{\langle #1 \rangle} 		
\def\dist(#1){\mathrm{dist}^{\Sigma}(#1)} 	
\def\tr{\mathrm{tr}}					
\def\nablaS{\nabla^{\Sigma}} 			
\def\intS{\int_{\Sigma}}				
\def\anF{\mathcal{F}_\gamma}			
\def\Rext{R_{\mathrm{ext}}^\gamma}	
\def\DH{d_{\mathcal{H}}}				

\title{Anisotropic extrinsic radius pinching for hypersurfaces and the stability of the Wulff shape}
\author{Toshimi Inoue}
\date{}							

\begin{document}
\maketitle

\begin{abstract}
We prove the Hasanis--Koutroufiotis type inequality for the anisotropic extrinsic radius of hypersurfaces in Euclidean space involving the anisotropic mean curvatures. We also study the equality case and proved that an almost extremal hypersurface must be close to the Wulff shape in the sense of the Hausdorff distance. 
\end{abstract}
\section{Introduction}
Let $X: \Sigma^n \to \RR^{n+1}$ be a closed, isometrically immersed hypersurface. The {\it extrinsic radius} of $\Sigma$ is defined as the smallest radius of balls containing $\Sigma$. It is well-known that the extrinsic radius of $\Sigma$ can be bounded from below in terms of the mean curvature. Namely, T. Hasanis and D. Koutroufiotis \cite{HK} showed that 
\begin{align}
R_\mathrm{ext} \|H\|_\infty \geq 1, \label{eq:isotropicHK}
\end{align}
where $R_\mathrm{ext}$ and $H$ denote the extrinsic radius and the mean curvature of $\Sigma$ respectively. For closed hypersurfaces, this inequality can be proved from an $L^2$-estimate of the radius as follows (see \cite{AGR}): 
\begin{align}
\left( \intS |H|^2 \right)^{\frac{1}{2}}\left( \intS |X-X_0|^2 \right)^{\frac{1}{2}} \geq \mathrm{Vol}(\Sigma), \label{eq:isotropicL2HK}
\end{align}
where $X_0$ is the center of mass of $\Sigma$ which is defined by $X_0=\frac{1}{\mathrm{Vol}(\Sigma)}\intS X$ and $\mathrm{Vol}(\Sigma)$ is the volume of $\Sigma$. 
Moreover, the equality holds in \eqref{eq:isotropicL2HK} (hence in \eqref{eq:isotropicHK}) if and only if $\Sigma$ is the $n$-dimensional sphere of radius $1/\|H\|_\infty$ centered at $X_0 \in \RR^{n+1}$. 

The aim of this paper is to obtain an anisotropic generalization of the extrinsic radius estimate \eqref{eq:isotropicL2HK} and quantitative and qualitative stability results proved in \cite{AGR, CG, R1} for the isotropic setting. 

To state our main results, we need some notations. Let $\gamma : \SS^{n} \to \RR_{>0}$ be a smooth, positive function satisfying the convexity condition
\begin{align}
A_\gamma = (\mathrm{Hess}^{\SS^n}_{\gamma} + \gamma I)_{\nu} > 0, 		\label{eq:convex}
\end{align}
for any $\nu \in \SS^n$. Here, $I$ denotes the identity operator on $T_{\nu}\SS^n$ and $>0$ means the positivity of self-adjoint operators. We consider the map given by

\begin{align}
	\begin{array}{rccl}
	\xi : &\SS^{n}	&\longrightarrow&	\RR^{n+1}\\
		&\nu		&\longmapsto&		\gamma(\nu)\nu+\nabla^{\SS^{n}}\gamma(\nu). 
	\end{array}
\end{align}
The image $W_\gamma = \xi(\SS^n)$ is called the {\it Wulff shape} with respect to $\gamma$. Note that $W_\gamma$ is a convex hypersurface in $\RR^{n+1}$ by the convexity consition \eqref{eq:convex}. 

The Wulff shape $W_\gamma$ can be seen as the "round sphere" for an anisotropic norm on $\RR^{n+1}$. Namely, if we introduce the {\it Minkowski norm} $\gamma^{*}: \RR^{n+1} \to \RR_{ \geq 0}$ by 
\begin{align}
\gamma^{*}(x) = \sup_{|z|\gamma(\frac{z}{|z|}) \leq 1} \inn<x, z>, 	\label{eq:gammadual}
\end{align}
then $W_\gamma$ can be represented as $W_\gamma = \{\gamma^{*} = 1\}$ (see \cite{Sch} for details). For a positive $s > 0$, we call the set $sW_\gamma = \{\gamma^{*} = s\}$ the {\it Wulff shape of radius $s$}.  We now define the {\it anisotropic extrinsic radius} $\Rext$ for a closed, isometrically immersed hypersurface $X: \Sigma^n \to \RR^{n+1}$ by
\begin{align}
\Rext = \inf_{x_0 \in \RR^{n+1}}\max_{p \in \Sigma}\gamma^{*}(X(p)-x_0). \label{eq:defRext}
\end{align}
We note that $\Rext$ is a natural anisotropic generalization of $R_\mathrm{ext}$. Indeed, we have that
\begin{align}
\Rext = \inf\{s > 0 | \Sigma \subset \mathrm{Int}(sW_\gamma) + x_0 \quad\text{for some $x_0 \in \RR^{n+1}$}\}
\end{align} 
Throughout of this paper, we let $X_0 \in \RR^{n+1}$ denote a point at which minimizes the right hand side of \eqref{eq:defRext}. 

The {\it anisotropic shape operator} $S_\gamma$ of $\Sigma$ is defined be $S_\gamma = A_\gamma \circ S$, where $S$ denotes the usual shape operator of $\Sigma$. We define the {\it anisotropic mean curvature} $H_\gamma$ by $H_\gamma=(-1/n)\tr{S_\gamma}$. It is known that the Wulff shape $W_\gamma$ is a stable constant anisotropic mean curvature hypersurface with $H_\gamma = 1$, like a round sphere in the isotropic setting (see for example \cite{K, P}). 

Finally, we define the $L^p$ norm of a function $f$ on $\Sigma$ by
\begin{align}
\|f\|_p = \left(\frac{1}{\anF(\Sigma)}\intS |f|^p\gamma(N) \right)^{\frac{1}{p}}, 
\end{align}
where $\anF(\Sigma) = \intS \gamma(N)$ is the {\it anisotropic surface energy} of $\Sigma$ and $N$ is the unit normal vector field along $\Sigma$. 

Our first result is the following anisotropic version of the extrinsic radius estimate. 
\begin{thm}\label{thm:anisoHK}
Let $X: \Sigma^n \to \RR^{n+1}$ be a closed, isometrically immersed hypersurface. Let $\gamma : \SS^{n} \to \RR_{>0}$ be a smooth positive function satisfying the convexity condition \eqref{eq:convex}. Then, it follows that 
\begin{align}
\|H_\gamma\|_2 \|\gamma^{*}(X-X_0)\|_2 \geq 1. 	\label{eq:L2HK}
\end{align}
In particular, the anisotropic extrinsic radius of $\Sigma$ satisfies 
\begin{align}
\|H_{\gamma}\|_{\infty} \Rext \geq 1. \label{eq:HKineq}
\end{align}
Moreover, equality occurs in \eqref{eq:L2HK} or \eqref{eq:HKineq} if and only if $\Sigma$ is the Wulff shape with respect to $\gamma$ of radius $\Rext$ up to translations. 
\end{thm}

A natural question related to the equality case in \eqref{eq:L2HK} is the following: If the equality almost holds in \eqref{eq:L2HK} (or \eqref{eq:HKineq}), is $\Sigma$ close to a rescaled Wulff shape $\|H_\gamma\|_2^{-1}W_\gamma$ in a certain sense? More precisely, we consider the following pinching condition for $p>2$ and $\eps>0$:
\begin{align}
\|H_\gamma\|_p \|\gamma^{*}(X-X_0)\|_2 \leq 1+\eps. 	\tag{$P_{p, \eps}$}\label{eq:pinching}
\end{align}

In recent years, many authors study generalizations of classical pinching results for geometric invariants of hypersurfaces to the anisotropic case. In \cite{dRG}, De Rosa and Gioffr\`{e} studied the nisotropic almost totally umbilical hypersurfaces and proved the stability of the Wulff shapes. More precisely, they proved that if the $L^p$ norm of the trace-free part of the anisotropic second fundamental form of a hypersurface $\Sigma$ is sufficiently small, then $\Sigma$ must be close to the Wulff shape in the Sobolev $W^{2, p}$ sense. Roth \cite{R2} used their results to prove that a convex hypersurface with almost constant anisotropic mean curvatures of the first and the second order must be close to the Wulff shape. Recently, Scheuer and Zhang \cite{SZ} studied the quantitative stability of Wulff shape for the anisotropic Heintze--Karcher inequality and the anisotropic Alexandrov theorem. 

Our next result shows that, when $\|S_{\gamma}\|_q$ is bounded for some $q>n$, the pinching condition \eqref{eq:pinching} implies that $\Sigma$ is close to $W_\gamma$ with respect to the Hausdorff distance. 
\begin{thm}\label{thm:mainthm}
Let $X: \Sigma^n \to \RR^{n+1}$ be a closed, isometrically immersed hypersurface. Let $\gamma : \SS^{n} \to \RR_{>0}$ be a smooth positive function satisfying the convexity condition. Let $q>n$, $p>2$, and $A>0$ be some real constants. Assume that the anisotropic shape operator satisfies $\anF(\Sigma)^{1/n}\|S_\gamma\|_q \leq A$. Then there exists some positive constants $C = C(n, p, q, A, \gamma)$ and $\alpha = \alpha(n, q)$ such that if $\Sigma$ satisfies \eqref{eq:pinching}, then we have
\begin{align}
\left\|\gamma^{*}(X-X_0)-\frac{1}{\|H_{\gamma}\|_2} \right\|_{\infty} \leq C\eps^{\alpha}\frac{1}{\|H_{\gamma}\|_2},		\label{eq:radpinch}
\end{align}
and for any $r \in [1, p)$ there exists some positive $D=D(n, p, q, r, A, \gamma)$ such that 
\begin{align}
\|H_\gamma-\|H_\gamma\|_2\|_r \leq D \eps^{\frac{\alpha(p-r)}{r(p-1)}}\|H_\gamma\|_2. 	\label{eq:mcpinch}
\end{align}

Moreover, given $\eps_0 > 0$, there exist a positive $\eps = \eps(n, p, q,  A, \gamma, \|H_\gamma\|_{\infty}, \eps_0)$ such that the pinching condition \eqref{eq:pinching} implies $\DH(\Sigma, \|H_\gamma\|_2^{-1}W_{\gamma}) < \eps_0$, where $\DH$ denotes the Hausdorff distance. 
\end{thm}

This paper is organized as follows. In Section 2, we give some necessary background on anisotropic geometry and  prove Theorem \ref{thm:anisoHK}. In the following sections, we consider the hypersurfaces satisfying the condition \eqref{eq:pinching}. We prove the inequalities \eqref{eq:radpinch} and \eqref{eq:mcpinch} in Section 3. In Section 4, we give a proof of the Hausdorff closeness of $\Sigma$ to the  rescaled Wulff shape. 

\section{Preliminaries}
Let $X: \Sigma^n \to \RR^{n+1}$ be a closed, isometrically immersed hypersurface and let $N$ be the unit normal vector field along $\Sigma$. Let $\inn<\cdot, \cdot>$ and $\nabla$ denote the canonical Riemannian metric and the connection on $\RR^{n+1}$ respectively. Let $\nablaS$ be the connection on $\Sigma$ with respect the induced Riemannian metric from $\inn<\cdot, \cdot>$. The shape operator $S$ of $\Sigma$ is a $(1, 1)$-tensor on $\Sigma$ defined by $Sv = \nabla_v N$. 

Let $\gamma: \SS^n \to \RR_{>0}$ be a smooth positive function on $\SS^n$ satisfying the convexity condition \eqref{eq:convex}. We define the {\it anisotropic shape operator} $S_\gamma$ by $S_\gamma = A_\gamma \circ S$. The {\it anisotropic mean curvature} $H_\gamma$ is given by $H_\gamma = -(1/n)\tr{S_\gamma}$. 

In \cite{HL}, He and Li proved that the anisotropic mean curvature satisfies the Hsiung--Minkowski type formula \cite{Hs} given by 
\begin{align}
\intS (\gamma(N) + H_\gamma \inn<X, N>) = 0. \label{eq:HMformula}
\end{align}
Such an integral formula plays an important role in the rigidity results involving anisotropic mean curvatures (see \cite{HL, K, P} for example). 

Let us now consider the Wulff shape $W_\gamma$ for $\gamma$. Let $\gamma^{*}$ be the dual of $\gamma$ defined by \eqref{eq:gammadual}. We extend $\gamma$ $1$-homogeneously to $\RR^{n+1}$ by letting 
\begin{align}
\gamma(x) = |x|\gamma \left( \frac{x}{|x|} \right)
\end{align}
As an immediate consequence of the definition of $\gamma^{*}$, we have the Fenchel inequality given by
\begin{align}
\inn<x, y> \leq \gamma^{*}(x)\gamma(y)	\label{eq:Fenchel}
\end{align}
for $x, y \in \RR^{n+1}$. Moreover, since $\gamma$ is the supporting function of $W_\gamma$, the equality holds in \eqref{eq:Fenchel} if and only if $x$ is perpendicular to the tangent plane of $W_\gamma$ at $\frac{y}{\gamma^{*}(y)} \in W_\gamma$. We now fix a point $x$ and let $\nu$ be the unit normal vector to $W_\gamma$ at $\frac{x}{\gamma^{*}(x)}$. Differentiating the function $G(x) = \gamma(\nu)\gamma^{*}(x)-\inn<\nu, x>$ as in \cite{N}, we can obtain the gradient of $\gamma^{*}$ as
\begin{align}
\nabla \gamma^{*} (x)= \frac{\nu}{\gamma(x)}. 	\label{eq:gradgamma}
\end{align}

\begin{proof}[Proof of Theorem \ref{thm:anisoHK}]
By the Hsiung--Minkowski formula \eqref{eq:HMformula} and the Fenchel inequality \eqref{eq:Fenchel}, we have
\begin{align}
1 = \left| \frac{1}{\anF(\Sigma)} \intS H_\gamma \inn<X-X_0, N> \right| \leq \frac{1}{\anF(\Sigma)} \intS |H_\gamma| \gamma^{*}(X-X_0) \gamma(N) \leq \|H_\gamma\|_2 \|\gamma^{*}(X-X_0)\|_2, 
\end{align}
which concludes the inequality. 

Assume the equality holds. Set $X_{\gamma^{*}}=\frac{X-X_0}{\gamma^{*}(X-X_0)}$. We have the equality of the Fenchel inequality \eqref{eq:Fenchel}, which implies that the unit normal $N$ is perpendicular to the tangent space of $W_\gamma$ at $X_{\gamma^{*}}$. Moreover, by \eqref{eq:gradgamma}, we have $\nablaS \gamma^{*}(X-X_0) = 0$, which implies that $\gamma^{*}(X-X_0)$ is constant. Therefore, we have $\Sigma = H_\gamma^{-1}W_\gamma +X_0$. 
\end{proof}

\section{Proof of \eqref{eq:radpinch} and \eqref{eq:mcpinch}}
For hypersurfaces, we have the following Michael--Simon Sobolev inequality \cite{MS}: 

\begin{align}
\left( \intS |f|^{\frac{n}{n-1}}\right)^{\frac{n-1}{n}} \leq C(n)\left( \intS |\nablaS f| + \intS |Hf| \right). \label{eq:Sobolev}
\end{align}

Let $\{e_i\}$ be an orthonormal frame along $\Sigma$ which diagonalizes $A_\gamma$ and set $a_i = \inn<A_\gamma e_i, e_i>$. By the Cauchy--Schwartz inequality, we have

\begin{align}
H&=\frac{1}{n}\sum_i \inn<Se_i, e_i>=\frac{1}{n}\sum_i \inn<A_\gamma^{-1}S_\gamma e_i, e_i> = \frac{1}{n}\sum_i a_i^{-1} \inn<S_\gamma e_i, e_i> \\
&\leq \frac{1}{n}\left(\sum_i a_i^{-2}\right)^{\frac{1}{2}}\left(\sum_i |\inn<S_\gamma e_i, e_i>|^2\right)^{\frac{1}{2}} \\
&\leq \frac{1}{\lambda}|S_\gamma|, 
\end{align}
where $\lambda$ is a positive constant given by 
\begin{align}
\lambda = \min_{\nu \in \SS^n, u \in \nu^{\perp}, |u|=1 }\inn<A_\gamma(\nu) u, u>. 
\end{align}
Combining this with \eqref{eq:Sobolev}, we have
\begin{align}
\|f\|_{\frac{n}{n-1}} \leq C(n, \gamma) \anF(\Sigma)^{\frac{1}{n}} (\|\nablaS f\|_1 + \||S_\gamma| f\|_1) \label{eq:anSobolev}. 
\end{align}

To obtain the inequality \eqref{eq:radpinch}, we prove the following estimate, which is an anisotropic version of \cite[Theorem 1.6]{AGR}. 

\begin{prop}\label{prop:Linftyradestimate}
Let $q > n$ be a real. There exists a constant $C = C(n, q, \gamma)> 0$ such that for any isometrically immersed hypersurface $X: \Sigma \to \RR^{n+1}$, we have 

\begin{align}
\|\gamma^{*}(X-X_0)-\|\gamma^{*}(X-X_0)\|_2\|_{\infty} \leq C(\anF(\Sigma)^{\frac{1}{n}}\|S_\gamma\|_2)^{\beta} \|\gamma^{*}(X-X_0)\|_2 \left(1-\frac{\|\gamma^{*}(X-X_0)\|_1}{\|\gamma^{*}(X-X_0)\|_2} \right)^{\frac{1}{2(1+\beta)}}, 
\end{align}
where $\beta = \frac{nq}{2(q-n)}$. 
\end{prop}

\begin{proof}
Up to translation, we may assume that $X_0=0$. We set $\varphi = |\gamma^{*}(X)-\|\gamma^{*}(X)\|_2|$. For a positive $a>0$, we have $|\nablaS \varphi^{2a}| \leq 2(\min{\gamma})^{-1}\varphi^{2a-1}$ by \eqref{eq:gradgamma}. Letting $f=\varphi^{2a}$ in \eqref{eq:anSobolev}, we have

\begin{align}
\|\varphi\|_{\frac{2an}{n-1}}^{2a} 
&\leq C(n, \gamma) \anF(\Sigma) (2a\|\varphi\|_{2a-1}^{2a-1} + \||S_\gamma| \varphi^{2a}\|_1) \\
&\leq C(n, \gamma) \anF(\Sigma) (2a\|\varphi\|_{2a-1}^{2a-1} + \|S_\gamma\|_q \|\varphi\|_{\frac{2aq}{q-1}}^{2a}) \\
&\leq C(n, \gamma) \anF(\Sigma) (2a\|\varphi\|_{\frac{(2a-1)q}{q-1}}^{2a-1} + \|S_\gamma\|_q \|\varphi\|_{\frac{2aq}{q-1}}^{2a}) \\
&\leq C(n, \gamma) \anF(\Sigma) (2a + \|S_\gamma\|_q \|\varphi\|_{\infty})\|\varphi\|_{\frac{(2a-1)q}{q-1}}^{2a-1}. \label{eq:iterr}
\end{align}

We set $\nu = \frac{n(q-1)}{(n-1)q}$ and $a = a_k\frac{q-1}{2q}+\frac{1}{2}$ where $a_{k+1} = \nu a_k + \frac{n}{n-1}$ and $a_0 = \frac{2q}{q-1}$. Plugging them into \eqref{eq:iterr} gives 

\begin{align}
\left(\frac{\|\varphi\|_{a_{k+1}}}{\|\varphi\|_\infty}\right)^{\frac{a_{k+1}}{\nu^{k+1}}} \leq \left\{ C(n, \gamma) \anF(\Sigma)^{\frac{1}{n}}\left( \frac{a_k\frac{q-1}{q}+1}{\|\varphi\|_\infty} + \|S_\gamma\|_q\right) \right\}^{\frac{n}{\nu^{p+1}(n-1)}} \left(\frac{\|\varphi\|_{a_{k}}}{\|\varphi\|_\infty}\right)^{\frac{a_{k}}{\nu^{k}}}
\end{align}

Since $q>n$ then $\nu$ and $\frac{a_k}{\nu^k}$ converges to $a_0 + \frac{qn}{q+n}$ and we have 

\begin{align}
1 &\leq  \left( \frac{\|\varphi\|_{a_0}}{\|\varphi\|_\infty} \right)^2 \prod_{k=0}^{\infty} \left\{ 2C(n, \gamma) \anF(\Sigma)^{\frac{1}{n}} a_k \left(\frac{1}{\|\varphi\|_\infty} + \|S_\gamma\|_q\right) \right\}^{\frac{1}{\nu^k}} \\
&\leq \left( \frac{\|\varphi\|_{a_0}}{\|\varphi\|_\infty}\right)^2 \prod_{k=0}^{\infty}a_k^{\frac{1}{\nu^k}} \left\{ 2C(n, \gamma) \anF(\Sigma)^{\frac{1}{n}} \left(\frac{1}{\|\varphi\|_\infty} + \|S_\gamma\|_q\right) \right\}^{\frac{\nu}{\nu-1}} \\
&= C(q, n, \gamma) \left( \frac{\|\varphi\|_{a_0}}{\|\varphi\|_\infty}\right)^2 \left\{\anF(\Sigma)^{\frac{1}{n}} \left(\frac{1}{\|\varphi\|_\infty} + \|S_\gamma\|_q\right) \right\}^{\frac{n(q-1)}{q-n}} \\
&\leq C(q, n, \gamma) \left( \frac{\|\varphi\|_{a_0}}{\|\varphi\|_\infty}\right)^{\frac{2(q-1)}{q}} \left\{\anF(\Sigma)^{\frac{1}{n}} \left(\frac{1}{\|\varphi\|_\infty} + \|S_\gamma\|_q\right) \right\}^{\frac{n(q-1)}{q-n}},  
\end{align}
hence we have 
\begin{align}
\|\varphi\|_\infty \leq C(n, q, \gamma) \left\{\anF(\Sigma)^{\frac{1}{n}} \left(\frac{1}{\|\varphi\|_\infty} + \|S_\gamma\|_q\right) \right\}^{\frac{nq}{2(q-n)}} \|\varphi\|_2. 
\end{align}
We set $\beta = \frac{nq}{2(q-n)}$. If $\|\varphi\|_\infty \geq \|S_\gamma\|_q^{-\frac{\beta}{1+\beta}}\|\varphi\|_2^{\frac{1}{1+\beta}}$, then

\begin{align}
\|\varphi\|_\infty &\leq C(q, n, \gamma) \left\{ \anF(\Sigma)^{\frac{1}{n}}\left( \frac{1}{\|\varphi\|_\infty} + \|S_\gamma\|_q\right)\right\}^\beta \|\varphi\|_2 \\
&\leq C(q, n, \gamma) (\anF(\Sigma)^{\frac{1}{n}}( \|S_\gamma\|_q^{\frac{\beta}{1+\beta}}\|\varphi\|_2^{-\frac{1}{1+\beta}} + \|S_\gamma\|_q))^\beta \|\varphi\|_2 \\
&\leq C(q, n, \gamma) (\anF(\Sigma)^{\frac{1}{n}}\|S_\gamma\|_q)^\beta ( \|S_\gamma\|_q^{-\frac{1}{1+\beta}} + \|\varphi\|_2^{\frac{1}{1+\beta}})^\beta \|\varphi\|_2^{\frac{1}{1+\beta}} \\
&\leq C(q, n, \gamma) (\anF(\Sigma)^{\frac{1}{n}}\|S_\gamma\|_q)^\beta (\|\gamma^{*}(X)\|_2^{\frac{1}{1+\beta}} + \|\varphi\|_2^{\frac{1}{1+\beta}})^\beta \|\varphi\|_2 \\
&\leq C(q, n, \gamma) (\anF(\Sigma)^{\frac{1}{n}}\|S_\gamma\|_q)^\beta \|\gamma^{*}(X)\|_2^{\frac{\beta}{1+\beta}} \|\varphi\|_2, 
\end{align}
where we have used $1\leq \|H_\gamma\|_2\|\gamma^{*}(X)\|_2 \leq \|S_\gamma\|_q \|\gamma^{*}(X)\|_2$. Since we have 
\begin{align}
\|\varphi\|_2^2 = \|\gamma^{*}(X)-\|\gamma^{*}(X)|_2\|_2^2 = 2\|\gamma^{*}(X)\|_2 \left(1-\frac{\|\gamma^{*}(X)\|_1}{\|\gamma^{*}(X)\|_2} \right), 
\end{align}
the desired inequality follows. 

If $\|\varphi\|_\infty \leq \|S_\gamma\|_q^{-\frac{\beta}{1+\beta}}\|\varphi\|_2^{\frac{1}{1+\beta}}$, the result follows immediately from the above expression of $\|\varphi\|_2$ and the fact that $\|S_\gamma\|_q\|\gamma^{*}(X)\|_2 \geq 1$. 
\end{proof}

\begin{proof}[Proof of \eqref{eq:radpinch}]
We may assume that $X_0=0$. From the Hsiung--Minkowski formula \eqref{eq:HMformula}, it follows that $1 \leq \|H_\gamma\|_p\|\gamma^{*}(X)\|_{\frac{p}{p-1}}$. By the H\"older inequality and the pinching condition \eqref{eq:pinching}, we have
\begin{align}
\|H_\gamma\|_p\|\gamma^{*}(X)\|_2 \leq 1+\eps \leq (1+\eps)\|H_\gamma\|_p\|\gamma^{*}(X)\|_{\frac{p}{p-1}} \leq (1+\eps) \|H_\gamma\|_p \|\gamma^{*}(X)\|_1^{1-\frac{2}{p}} \|\gamma^{*}\|_2^{\frac{2}{p}}, 
\end{align}
hence
\begin{align}
1-\frac{\|\gamma^{*}(X)\|_1}{\|\gamma^{*}(X)\|_2} \leq 1-\frac{1}{(1+\eps)^{\frac{p}{p-2}}} \leq \frac{p}{p-2}2^{\frac{2}{p-2}} \eps. 
\end{align}
Combining this inequality with Proposition \ref{prop:Linftyradestimate}, we obtain
\begin{align}
\|\gamma^{*}(X)-\|\gamma^{*}(X)\|_2\|_\infty &\leq C(n, p, q, \gamma)(\anF(\Sigma)\|S_\gamma\|_q)^\beta \|\gamma^{*}(X)\|_2 \eps^{\frac{1}{2(1+\beta)}} \\
&\leq C(n, p, q, \gamma)A^\beta \frac{1}{\|H_\gamma\|_2}\eps^{\frac{1}{2(1+\beta)}}, 
\end{align}
Here, we used $\|H_\gamma\|_2\|\gamma^{*}(X)\|_2 \leq 1+\eps \leq 2$ to get the second inequality. 

Letting $\alpha = \frac{1}{2(1+\beta)}$, we obtain
\begin{align}
\left\|\gamma^{*}(X) - \frac{1}{\|H_\gamma\|_2}\right\|_\infty
&\leq \|\gamma^{*}(X) - \|\gamma^{*}(X)\|_2\|_\infty + \left\| \|\gamma^{*}(X)\|_2 - \frac{1}{\|H_\gamma\|_2}\right\|_\infty \\
&\leq CA^{\beta}\frac{\eps^\alpha}{\|H_\gamma\|_2} + \frac{\eps}{\|H_\gamma\|_2} \\
&\leq C(n, p, q, A, \gamma) \frac{\eps^\alpha}{\|H_\gamma\|_2}. 
\end{align}
\end{proof}

\begin{proof}[Proof of \eqref{eq:mcpinch}]
By the Hsiung--Minkowski formula \eqref{eq:HMformula} and the Fenchel inequality \eqref{eq:Fenchel}, we have
\begin{align}
1 \leq \frac{1}{\anF(\Sigma)}\intS |H_\gamma|\gamma^{*}(X)\gamma(N). 
\end{align}
Using this and \eqref{eq:pinching}, we have
\begin{align}
\left\| \frac{|H_\gamma|}{\|H_\gamma\|_2^2} - \gamma^{*}(X) \right\|_2^2
&= \|\gamma^{*}(X)\|_2^2 + \frac{1}{\|H_\gamma\|_2^2} - \frac{2}{\|H_\gamma\|_2^2} \frac{1}{\anF(\Sigma)}\intS |H_\gamma|\gamma^{*}(X)\gamma(N) \\
&\leq \|\gamma^{*}(X)\|_2^2 - \frac{1}{\|H_\gamma\|_2^2} \\
&\leq \left( 1-\frac{1}{(1+\eps}^2 \right) \|\gamma^{*}(X)\|_2^2 \\
&\leq 3\eps \|\gamma^{*}(X)\|_2^2, 
\end{align}
hence
\begin{align}
\|H_\gamma^2 - \|H_\gamma\|_2^2\|_1 &\leq \|H_\gamma^2 - \gamma^{*}(X)^2\|H_\gamma\|_2^4\|_1 + \|\gamma^{*}(X)^2\|H_\gamma\|_2^4 - \|H_\gamma\|_2^2\|_1 \\
&= \|H_\gamma\|_2^4\left(\left\| \frac{|H_\gamma|^2}{\|H_\gamma\|_2^2} - \gamma^{*}(X)^2 \right\|_1 +  \left\|  \gamma^{*}(X)^2 - \frac{1}{\|H_\gamma\|_2^2} \right\|_1 \right) \\
&\leq \|H_\gamma\|_2^4\left(\left\| \frac{|H_\gamma|}{\|H_\gamma\|_2} - \gamma^{*}(X) \right\|_2 \left\| \frac{|H_\gamma|}{\|H_\gamma\|_2} + \gamma^{*}(X) \right\|_2 +  CA^\beta \frac{\eps^\alpha}{\|H_\gamma\|_2}\left\| \gamma^{*}(X) + \frac{1}{\|H_\gamma\|_2} \right\|_1 \right)\\
&\leq \|H_\gamma\|_2^4 \left(\sqrt{3\eps}\|\gamma^{*}(X)\|_2 + CA^\beta \frac{\eps^\alpha}{\|H_\gamma\|_2} \right) \left(\|\gamma^{*}(X)\|_2 + \frac{1}{\|H_\gamma\|_2}\right) \\
&\leq C(n, p, q, A, \gamma) \eps^\alpha \|H_\gamma\|_2^2. 
\end{align}
Therefore, we obtain 
\begin{align}
\||H_\gamma| - \|H_\gamma\|_2\|_1\leq \frac{\||H_\gamma|^2 - \|H_\gamma\|_2^2\|_1}{\|H_\gamma\|_2} \leq C \eps^\alpha \|H_\gamma\|_2
\end{align}
Moreover, we have
\begin{align}
\||H_\gamma| - \|H_\gamma\|_2\|_p \leq 2\|H_\gamma\|_p \leq 2\|H_\gamma\|_2 \|\gamma^{*}(X)\|_2\|H_\gamma\|_p \leq 4 \|H_\gamma\|_2
\end{align}
by \eqref{eq:pinching}. Hence, for any $r \in [1, p)$, we obtain
\begin{align}
\||H_\gamma| - \|H_\gamma\|_2\|_r &\leq \||H_\gamma| - \|H_\gamma\|_2\|_1^{\frac{p-r}{r(p-1)}} \||H_\gamma| - \|H_\gamma\|_2\|_p^{\frac{p(r-1)}{r(p-1)}} \\
&\leq D(n, p, q, r, A, \gamma)\eps^{\frac{\alpha(p-r)}{r(p-1)}} \|H_\gamma\|_2. 
\end{align}
\end{proof}
\section{Proof of the Hausdorff closeness}
In this section, we prove the Hausdorff closeness in Theorem \ref{thm:mainthm}. To prove this, we need an anisotropic version of the lemma due to B. Colbois and J. -F. Grosjean \cite[Lemma 3.2]{CG}. 

\begin{lem}\label{lem:curvebound}
For any $R>0$ and $\rho \in (0, 1)$, there exists a positive $\eta=\eta(R, \rho, n, \gamma) >0$ satisfying the following property: Let $z_0 \in RW_\gamma$ and let $\nu_0$ be the outer unit normal vector to $RW_\gamma$ at $z_0$. Let $X: \Sigma \to \RR^{n+1}$ be a closed hypersurface isometrically immersed in $((R+\eta)W_\gamma \setminus (R-\eta)W_\gamma) \setminus B_{\rho}(z_0)$.  If there exists a point $p \in \Sigma$ with $\inn<X(p), \nu_0> >0$, then there exists a point $q \in \Sigma$ such that $|H_\gamma(q)| > \frac{\lambda}{2n\eta}$. 
\end{lem}

Before proving Lemma \ref{lem:curvebound}, we prepare some notations. For $\nu \in \SS^n$ and $t>0$, we set $\Pi_t(\nu) = \nu^{\perp}+t\nu$, where $\nu^{\perp}$ denote the $n$-dimensional subspace of $\RR^{n+1}$ which is  perpendicular to $\nu$. Let $z_0=R\xi(\nu) \in RW_\gamma$. 

We now let $t_0(\nu) = \max\{t>0 | \Pi_t(\nu) \cap (\partial B_{\rho}(z_0) \cap RW_\gamma) \neq \emptyset \}$. Note that for every $t \in [0, t_0(\nu)]$, the set $W(\nu, t) = RW_\gamma \cap \Pi_t(\nu)$ is convex. Considering $W(\nu, t)$ as a convex hypersurface in $\RR^n$, we let $S^{W(\nu, t)}$ and $\kappa_i^{W(\nu, t)}$ be the shape operator of $W(\nu, t)$ and the $i$-th principal curvature of $W(\nu, t)$ respectively. Set $H_\gamma^{W(\nu, t)} = \tr(A_\gamma S^{W(\nu, t})$. Since the function $h^{\nu}:[0, t_0(\nu)] \ni t \mapsto \max_{W(\nu, t)}H_\gamma^{W(\nu, t)} \in \RR_{>0}$ is non-decreasing we have $h^{\nu}(t_0(\nu)) = \max h^{\nu}$. We now define 
\begin{align}
H_0 = \max_{\nu \in \SS^n}h^{\nu}(t_0(\nu)) \in (0, \infty). 
\end{align}
Similarly, we set 
\begin{align}
\kappa_0 = \max_{\nu \in \SS^n} \max\{\kappa_i^{W(\nu, t)}(x) | t \in [0, t_0(\nu)], x \in W(\nu, t), 1 \leq i \leq n-1\}. 
\end{align}

\begin{proof}[Proof of Lemma \ref{lem:curvebound}]
For $\eta \in (0, \rho)$ and $t \in [0, t_0(\nu_0)]$, consider the family of smooth maps
\begin{align}
	\begin{array}{rccl}
	\Phi_{\eta, t} : &W(\nu_0, t) \times \SS^{1}	&\longrightarrow&	\RR^{n+1}\\
		&(x, \theta)		&\longmapsto&		x - \eta \cos{\theta} N(x) +\eta \sin{\theta}\nu_0 + t\nu_0, 
	\end{array} 
\end{align}
where $N(x)$ denotes the unit normal vector of $W(\nu_0, t) \subset \RR^{n}$ at $x$. Let $S_{\eta, t}$ denote the image of $\Phi_{\eta, t}$. 

We calculate the curvature of $S_{\eta, t}$. Let $\{e_i\}_{i=1}^{n-1}$ be the orthonormal frame which diagonalize $S^{W(\nu_0, t)}$ at $x$. We have
\begin{align}
\Phi_i &= d\Phi_{\eta, t}(e_i) = (1-\eta \cos{\theta} \kappa_i^{W(\nu_0, t)})e_i, \\
\Phi_\theta &= d\Phi_{\eta, t}(\partial_\theta) = \eta (\sin{\theta} N(x) + \cos{\theta}\nu_0). 
\end{align}
Note that there exists a posit
ive $\rho_0 \in (0, \rho)$ such that $S_{\eta, t}$ is an embedded hypersurface for $\eta \in (0, \rho_0)$. Since the outer unit normal vector field of $S_{\eta, t}$ is given by $\overline{N} = -\cos{\theta} N + \sin{\theta} \nu_0$, the anisotropic shape operator $S_\gamma^{S_{\eta, t}}$ of $S_{\eta, t}$ can be calculated by 
\begin{align}
S_\gamma^{S_{\eta, t}} e_i = -\cos{\theta} A_\gamma S^{W(\nu_0, t)} e_i, \quad S_\gamma^{S_{\eta, t}} \partial_\theta = -\eta^{-1}A_\gamma \Phi_\theta. 
\end{align}
Since $\inn<\Phi_i, \Phi_j> = \delta_{ij}$ and $\inn<\Phi_i, \Phi_\theta> = 0$, the anisotropic mean curvature $H_\gamma^{S_{\eta, t}}$ is given by 
\begin{align}
H_\gamma^{S_{\eta, t}} &= -\frac{1}{n}\tr{S_\gamma^{S_{\eta, t}}} \\
&= -\frac{1}{n} \frac{\inn<S_\gamma^{S_{\eta, t}} \partial_\theta, \partial_\theta>}{|\Phi_\theta|^2} -\frac{1}{n} \sum_{i=1}^{n-1}\frac{\inn<S_\gamma^{S_{\eta, t}} e_i, e_i>}{|\Phi_i|^2} \\
&= \frac{\inn<A_\gamma \Phi_\theta, \Phi_\theta>}{n\eps |\Phi_\theta|^2} -\frac{1}{n}\sum_{i}^{n-1}\frac{\cos{\theta}\inn<A_\gamma S^{W(\nu_0, t)} e_i, e_i>}{(1-\eta \cos{\theta} \kappa_i^{W(\nu_0, t)})^2}\\
&\geq  \frac{\lambda}{n\eta} -\frac{(n-1)H_0}{n(1-\eta \kappa_0)^2}
\end{align}

We now let $\eta = \min \left\{ \frac{1}{2\kappa_0}, \frac{\lambda}{8(n-1)H_0}, \rho_0\right\}$, then
\begin{align}
H_\gamma^{S_{\eta, t}} \geq \frac{\lambda}{2n\eta}
\end{align}

Since there exists a point $p \in \Sigma$ so that $\inn<X(p), \nu_0>$ by assumption, we can find $t \in [0, t_0(\nu_0)]$ and a point $q \in \Sigma$ which is a contact point with $S_{\eta, t}$. Therefore $|H_\gamma(q)| \geq \frac{1}{2n\eta}$. 
\end{proof}

\begin{proof}[Hausdorff closeness in Theorem \ref{thm:mainthm}]
Let $B_{\eps}(\|H_\gamma\|_2^{-1}W_\gamma)$ denote the $\eps$-neighborhood of $\|H_\gamma\|_2^{-1}W_\gamma$. Assume a positive $\eps_0>0$ is given. If $\Sigma$ satisfies \eqref{eq:pinching} for a small $\eta>0$, by \eqref{eq:radpinch}, it follows for a small $\eps \in (0, \eps_0)$ that $\Sigma \subset B_{\eps}(\|H_\gamma\|_2^{-1}W_\gamma) \subset B_{\eps}(\|H_\gamma\|_2^{-1}W_\gamma)$. 

Assume $\Sigma \subset B_{\eps}(\|H_\gamma\|_2^{-1}W_\gamma) \setminus B_{\eps_0}(y)$ occurs for some $y \in \|H_\gamma\|_2^{-1}W_\gamma$. Choosing $\eps$ so small that $\eps < \min\{\eta, n/\|H_\gamma\|_\infty \}$, where $\eta$ is a positive as in Lemma \ref{lem:curvebound}, it follows from Lemma \ref{lem:curvebound} that there exists a point $q \in \Sigma$ such that $H_\gamma(q) > \|H_\gamma\|_\infty/2$, which is a contradiction. 

Hence, we have $\DH(\Sigma, \|H_\gamma\|_2^{-1}W_\gamma) \leq \eps_0$. 
\end{proof}

\end{document}